\documentclass[11pt,a4paper, reqno]{amsart}

\usepackage{amssymb,mathrsfs,enumitem,amsmath,amsthm,bbold}
\usepackage{amsaddr}
\usepackage[mathscr]{euscript}
\usepackage{color}

\usepackage[all]{xy}
\definecolor{Chocolat}{rgb}{0.36, 0.2, 0.09}
\definecolor{BleuTresFonce}{rgb}{0.215, 0.215, 0.36}
\definecolor{EgyptianBlue}{rgb}{0.06, 0.2, 0.65}

\usepackage[colorlinks,final,hyperindex]{hyperref}
\hypersetup{citecolor=BleuTresFonce, urlcolor=EgyptianBlue, linkcolor=Chocolat}

\usepackage{fourier}

\newtheorem*{itheorem}{Main conjecture}
\newtheorem{theorem}{Theorem}[section]
\newtheorem{corollary}[theorem]{Corollary}

\newtheorem{proposition}[theorem]{Proposition}

\theoremstyle{definition}
\newtheorem{example}[theorem]{Example}
\newtheorem{remark}[theorem]{Remark}
\newtheorem{definition}[theorem]{Definition}

\DeclareMathAlphabet{\pazocal}{OMS}{zplm}{m}{n}

\def\calB{\pazocal{B}}
\def\calC{\pazocal{C}}

\def\calP{\pazocal{P}}

\DeclareMathOperator{\Lie}{Lie}
\DeclareMathOperator{\Ass}{Ass}
\DeclareMathOperator{\Com}{Com}
\DeclareMathOperator{\Jord}{Jord}
\DeclareMathOperator{\SJord}{SJord}
\DeclareMathOperator{\TAG}{TAG}
\DeclareMathOperator{\Inner}{Inner}

\DeclareMathOperator{\Vect}{\mathsf{Vect}}

\DeclareMathOperator{\TKK}{\mathfrak{T}}

\DeclareMathOperator{\Tor}{Tor}

\DeclareMathOperator{\Mat}{Mat}

\usepackage[backend=bibtex,
    sorting=anyt,
    isbn=false,
    url=false,
    doi=false,
    maxbibnames=100
    ]{biblatex}
\DeclareMathAlphabet{\mathbbold}{U}{bbold}{m}{n}

\def\k{\mathbbold{k}}

\begin{document}

\title[The three graces in the Tits--Kantor--Koecher category]{The three graces in the Tits--Kantor--Koecher category}

\author[Vladimir Dotsenko]{Vladimir Dotsenko}

\address{Institut de Recherche Math\'ematique Avanc\'ee, UMR 7501, Universit\'e de Strasbourg et CNRS, 7 rue Ren\'e-Descartes, 67000 Strasbourg, France}

\email{vdotsenko@unistra.fr}

\author[Iryna Kashuba]{Iryna Kashuba}

\address{International Center of Mathematics, Southern University of Science and Technology, 1088 Xueyuan Avenue, Xili, Nanshan District, Shenzhen, 518055, China}

\email{kashuba@sustech.edu.cn}


\date{}

\begin{abstract}
A metaphor of Loday describes Lie, associative, and commutative associative algebras as ``the three graces'' of the operad theory. In this article, we study the three graces in the category of $\mathfrak{sl}_2$-modules that are sums of copies of the trivial and the adjoint representation. That category is not symmetric monoidal, and so one cannot apply the wealth of results available for algebras over operads. Motivated by a recent conjecture of the second author and Mathieu, we embark on the exploration of the extent to which that category ``pretends'' to be symmetric monoidal. To that end, we examine various homological properties of free associative algebras and free associative commutative algebras, and study the Lie subalgebra generated by the generators of the free associative algebra.
\end{abstract}

\maketitle


\section{Introduction}

For $\k$-linear algebras, at least over a field of characteristic zero, it is nowadays sufficiently standard to place them in the context of algebras over an 
operad in a symmetric monoidal category.
This paper is concerned with one situation where this intuition fails in a very notable way, but many features that operadic algebras exhibit are nevertheless present. 

Let us be a bit more specific about it. In \cite{MR4235202}, the second author and Mathieu proposed a conjecture which, if true, would lead to important new insight into the free Jordan algebra in several generators. The universe in which the story unfolds is the category $\TKK$ 
whose objects are completely reducible $\mathfrak{sl}_2$-modules that decompose as direct sums of copies of trivial modules and adjoint modules.  We can talk about Lie algebras in~$\TKK$, which are objects of~$\TKK$ which are Lie algebras whose Lie bracket is $\mathfrak{sl}_2$-equivariant, and even about free Lie algebras in~$\TKK$. It is proved in \cite{MR4235202} that the free Lie algebra in~$\TKK$ generated by $n$ copies of the adjoint representation can be described via a version of the celebrated Tits--Kantor--Koecher construction \cite{MR0146231,MR0175941,MR214631} due to Allison and Gao \cite{MR1402322,MR3169564}: it is the Lie algebra obtained by that construction from the free (non-unital) Jordan algebra on $n$ generators.

In this paper, we start a systematic study of the following conjecture (in the case where $U$ is the direct sum of several copies of the adjoint representation, this is the central conjecture of \cite{MR4235202}).

\begin{itheorem}
For an object $U$ of the category~$\TKK$, let us consider the homology with trivial coefficients of the free Lie algebra in~$\TKK$ generated by~$U$ with its natural $\mathfrak{sl}_2$-action, and let us truncate it by quotienting out all irreducible $\mathfrak{sl}_2$-submodules different from the trivial and the adjoint representation. That truncation is concentrated in degrees $0$ and $1$.
\end{itheorem}

In a sense, this suggests that in certain aspects the category $\TKK$ ``pretends'' to be a symmetric monoidal category. Indeed, for any symmetric monoidal category~$\calC$ containing the symmetric monoidal category of vector spaces over a field~$\k$ of zero characteristic as a full monoidal subcategory, the homology of any free Lie algebra in~$\calC$ is concentrated in degrees $0$ and $1$, a result which is essentially equivalent to the Koszul property of the Lie operad \cite[Th.~4.2.5]{MR1301191}. Thus, if $\TKK$ were the quotient of the symmetric monoidal category of $\mathfrak{sl}_2$-modules by a monoidal ideal, the main conjecture would follow automatically. 

In this paper, we prove that the obvious analogue of the main conjecture does not hold for free associative commutative algebras in~$\TKK$ and holds for free associative algebras in~$\TKK$. We also discuss the main conjecture for free Lie algebras in~$\TKK$, and prove a somewhat surprising result stating  that for the free associative algebra in~$\TKK$ generated by $n$ copies of the adjoint representation, the Lie subalgebra generated by its generators is the Tits--Kantor--Koecher construction of the free \emph{special} Jordan algebra on $n$ generators. 

\section{Preliminaries and recollections}

For simplicity we shall work over a field $\k$ of zero characteristic, though some of our results are available in greater generality. To simplify some formulas, we shall slightly abuse notation for free algebras: our $\Com(V)$ and $\Ass(V)$ will be the free \emph{unital} commutative associative algebra and the free \emph{unital} associative algebra, which, in classical terms, are, respectively, the symmetric algebra $S(V)$ and the tensor algebra $T(V)$. In all other aspects, our operadic conventions correspond to those of the monograph \cite{MR2954392}. For $n\ge 0$, we shall denote by $L(n)$ the irreducible $\mathfrak{sl}_2$-module of the highest weight $n$. In particular, $L(0)$ is the trivial module whose basis element we shall often denote by $x$, and $L(2)$ is the adjoint module $\mathfrak{sl}_2$ with the usual basis $e,f,h$. We shall normalise the Killing form of $\mathfrak{sl}_2$ so that its nonzero values are $K(e,f)=K(f,e)=\frac12$, $K(h,h)=1$. 

\subsection{The Tits--Kantor--Koecher category}

As mentioned in the introduction, this paper studies algebras in the category $\TKK$ whose objects are completely reducible $\mathfrak{sl}_2$-modules that decompose as direct sums of copies of trivial modules and adjoint modules, and whose morphisms are $\mathfrak{sl}_2$-module morphisms. In other words, each object of $\TKK$ is of the form 
 \[
L(0)\otimes A\oplus L(2)\otimes B, 
 \]
where $A$ and $B$ are vector spaces of multiplicities of the corresponding modules. We shall call this category the \emph{Tits--Kantor--Koecher category}.

Let $\calP$ be an operad in $\Vect$. Since objects of $\TKK$ are vector spaces, we can talk about $\calP$-algebras in $\TKK$ in the following straightforward way.

\begin{definition}
A $\calP$-algebra in $\TKK$ is an object $U$ of $\TKK$ equipped with a map 
 \[
\calP(U)\to U
 \]
which is a $\calP$-algebra structure in $\Vect$ and, additionally, a morphism of $\mathfrak{sl}_2$-modules for the obvious extension of the $\mathfrak{sl}_2$-module structure on $U$ to $\calP(U)$. All $\calP$-algebras in $\TKK$ form a category, where morphisms are maps that respect both the algebra structure and the $\mathfrak{sl}_2$-module structure. 
\end{definition}

Moreover, we can talk about free $\calP$-algebras in $\TKK$.

\begin{definition}
The free $\calP$-algebra in $\TKK$ generated by an object $U$, denoted by $\calP^{\TKK}(U)$ is the left adjoint of the forgetful functor from the category of $\calP$-algebras in $\TKK$ to the category $\TKK$. 
\end{definition}

Existence of free algebras is justified by the following obvious result for which we omit the proof. 

\begin{proposition}
The free algebra $\calP^{\TKK}(U)$ is isomorphic to the quotient of the free algebra $\calP(U)$ by the ideal generated by all ``wrong'' isotypic components for the $\mathfrak{sl}_2$-action (components of the form $L(k)$ for $k\ne 0,2$). 
\end{proposition}

At this point, a crucial warning is an order. Since the category $\TKK$ is not symmetric monoidal (it is a quotient of the symmetric monoidal category of all completely reducible $\mathfrak{sl}_2$-modules supported at finitely many weights by a subcategory which is not a monoidal ideal), $\calP$-algebras in $\TKK$, even though they have an operad behind their definition, are not algebras over an operad! Of course, what they really are is algebras over the free algebra monad $\calP^{\TKK}$. However, the free $\calP$-algebra in $\TKK$ generated by an object $U$, while defined in a sufficiently straightforward way, is not obtained from $U$ by any tractable standard formula.

\subsection{Some recollections on Jordan algebras}

Let us briefly recall some of the results on Jordan algebras and their relationship to Lie algebras in $\TKK$; we refer the reader to \cite{MR4235202} for details. 

A \emph{Jordan algebra} is a commutative not necessarily associative algebra $J$ satisfying the identity $(x^2y)x=x^2(yx)$ for all $x,y\in J$. It is known that for every Jordan algebra $J$, the associator $(x,y,z)=(xy)z-x(yz)$ is a derivation with respect to its argument $y$. A derivation obtained in this way for some $x,z\in J$ is denoted $\partial_{x,z}$ and is called an \emph{inner derivation} of $J$. It is known that the vector space $\Inner(J)$ of all inner derivations is a subalgebra of the Lie algebra of all derivations of $J$ \cite{MR668355}. 

Over a field of characteristic zero, every identity is equivalent to a multilinear identity, and in particular the Jordan identity is equivalent to the multilinear identity
 \[
((xy)z)t+
((yt)z)x+
((xt)z)y=
(xy)(zt)+
(xz)(yt)+
(xt)(yz).
 \] 
This means that Jordan algebras can be described as algebras over an operad, which we shall denote by $\Jord$. It is well known that there is a morphism of operads $\Jord\to\Ass$ sending the generator of the Jordan operad to the symmetrization of the associative product; in more classical terms, for every associative algebra, the operation
\begin{equation}\label{eq:SJordan}
a\circ b=\frac12(ab+ba)
\end{equation}
satisfies the Jordan identity. This morphism has a kernel; in other words, Jordan algebras arising as subalgebras of associative algebras in the way indicated above, always satisfy some extra identities \cite{MR0108524,MR0186708}. The image of that morphism is denoted $\SJord$ and called the special Jordan operad. As for any class of algebras over an operad, one can consider the free Jordan algebra generated by a vector space $V$, denoted by $\Jord(V)$, and the free special Jordan algebra generated by a vector space $V$, denoted by $\SJord(V)$. 
 
\begin{remark}
One important property of the Jordan operad that was established and meaningfully used by the second author and Mathieu \cite{MR4235202} states that it is \emph{cyclic}, which means that the natural $S_n$-action on its $n$-th component extends to a $S_{n+1}$-action (we refer the reader to the foundational paper \cite{MR1358617} for further details on cyclic operads). In fact, in parallel to the long conceptual proof of that result in \cite{MR4235202}, there exists a very short argument proving it: 
inspecting the Jordan identity, one notes that the $S_4$-module generated by it is preserved by the $S_5$-action arising from the standard cyclic operad structure on the free operad with one generator, and hence the Jordan operad is cyclic, being a quotient of the cyclic operad by a cyclically invariant ideal. By a similar argument, the special Jordan operad is cyclic, being a suboperad of a cyclic operad generated by a cyclically invariant subspace of the space of generators. 
\end{remark}

There is a relationship between Jordan algebras and Lie algebras in $\TKK$ going back to the work of Tits \cite{MR0146231} who noted that for each Lie algebra $\mathfrak{g}$ in $\TKK$ one can define a Jordan algebra structure on the vector space of weight $2$ elements in~$\mathfrak{g}$ by the formula $x,y\mapsto \frac12[x,f(y)]$, where $f(-)$ refers to the action of $f\in\mathfrak{sl}_2$ on~$\mathfrak{g}$. This construction is clearly a functor from the category $\TKK$ to the category of Jordan algebras, which we shall call the \emph{Tits functor}. Tits also noticed that there is a construction in the opposite direction, which however is not functorial: to a Jordan algebra $J$, one can associate the object $L(0)\otimes\Inner(J)\oplus L(2)\otimes J$ of $\TKK$, equipped with a Lie algebra structure defined as follows: for two elements of $L(0)\otimes\Inner(J)\cong\Inner(J)$, their Lie bracket is the Lie bracket of derivations, for an element of $L(0)\otimes\Inner(J)\cong\Inner(J)$ and an element of $L(2)\otimes J$, their Lie bracket comes from the action of $\Inner(J)$ on $L(2)\otimes J$, and finally for $u_1\otimes z_1,u_2\otimes z_2\in L(2)\otimes J$, one defines
 \[
[u_1\otimes z_1,u_2\otimes z_2]:=K(u_1,u_2)\partial_{z_1,z_2}+[u_1,u_2]\otimes (z_1z_2).
 \]
That construction was further investigated and generalized by Kantor \cite{MR0175941} and Koecher \cite{MR214631}, and is usually referred to as the \emph{Tits--Kantor--Koecher construction}. 

A functorial version of the Tits--Kantor--Koecher construction was found by Allison and Gao, see \cite{MR1402322,MR3169564}. It is defined as follows. Given a Jordan algebra $J$, we define the \emph{Tits--Allison--Gao functor}
 \[
\TAG(J):=L(0)\otimes \calB(J)\oplus L(2)\otimes J,
 \]
where $\calB(J):=\Lambda^2(J)/\k\{z\wedge z^2\colon z\in J\}$. Note that $\partial_{z,z^2}(y)=(z,y,z^2)=0$ due to the Jordan identity, so 
the natural map $\Lambda^2(J)\to \Inner(J)$ sending $x\wedge y$ to $\partial_{x,y}$ induces a surjective map $\calB(J)\to \Inner(J)$, and $\calB(J)$ can be viewed as a natural functorial replacement of the non-functorial $\Inner(J)$ in the Tits--Kantor--Koecher construction. The map $\calB(J)\to \Inner(J)$ leads to an action of $\calB(J)$ on $J$ by derivations, and one can show that if one extends it to $\Lambda^2(J)$, that action induces a Lie algebra structure on $\calB(J)$, thus leading to a Lie algebra structure on $\TAG(J)$ analogous to that of the Tits--Kantor--Koecher construction. It is shown in \cite{MR4235202} that the Tits functor and the Tits--Allison--Gao functor form a pair of adjoint functors.  

\subsection{Gr\"obner--Shirshov bases}

In this article, we make extensive use of Gr\"obner--Shirshov bases in all kinds of algebras that we consider. The purpose of such bases is to consider particularly useful systems of generators of ideals in free algebras, allowing one to have normal forms for elements in the quotient by an ideal.
We note that in different types of algebras, one encounters two qualitatively different situations. For commutative associative algebras and associative algebras, free algebras admit bases (of commutative monomials and of words, respectively), for which the product of two basis elements is always another basis element, whereas for Lie algebras and superalgebras, it is not possible to introduce a basis for which the Lie bracket of two basis elements is always a basis element, so that much more intricate considerations are required. To distinguish between such situations, we shall talk about \emph{Gr\"obner bases} in contexts of the first kind (and call the corresponding basis elements \emph{monomials}) and about \emph{Shirshov bases} in context of the second kind, though, as we see, Shirshov bases of ideals in Lie algebras and superalgebras can be related to Gr\"obner bases of ideals in the corresponding universal enveloping algebras. In this section, we only give very brief recollections, referring the reader to \cite{MR1700511,MR1733167,MR3642294,MR1360005} for details. 

Let us begin with the combinatorially more transparent case of commutative associative algebras and associative algebras. A total order of monomials in the free algebra is said to be \emph{admissible} if it is a well-order, and the product is an increasing function of its arguments: replacing one of the monomials in the product by a greater one increases the result. Given an admissible order of the free algebra, one can define a \emph{Gr\"obner basis} of an ideal $I$ as a subset $G\subset I$ for which the leading monomial of every element of $I$ is divisible by a leading monomial of an element of~$G$. The primary reason to look for Gr\"obner bases is dictated by considerations of linear algebra: a Gr\"obner basis for an ideal gives extensive information on the quotient modulo $I$. A monomial is said to be \emph{normal} with respect to $G$ if it is not divisible by any of the leading monomials of elements of~$G$. It is easy to show that the normal monomials with respect to any set of generators of an ideal $I$ always form a spanning set of the quotient modulo $I$. One can prove that $G$ is a Gr\"obner basis if and only if the cosets of monomials that are normal with respect to $G$ form a basis of the quotient modulo $I$.

In the Lie case, we shall mostly encounter Lie superalgebras, and so we discuss Shirshov bases in that generality. We start with explaining what kind of monomial bases in free algebras we consider. Suppose that $X$ is a set equipped with a well-order. We can consider the free monoid $\langle X\rangle$ generated by $X$, consisting of all words in the alphabet $X$ with the associative product of each two elements given by concatenation. We shall moreover assume that there is a parity function $X\to \mathbb{Z}/2\mathbb{Z}=\{0,1\}$ allowing to write $X=X_0\sqcup X_1$; we extend the parity to $\langle X\rangle$ additively, so that we can talk about even and odd words in the alphabet~$X$. A non-empty word $w$ is said to be a \emph{Lyndon--Shirshov word} if it is the strictly largest one (with respect to the graded lexicographic order of $\langle X\rangle$ induced by the order of $X$) among all its cyclic shifts. Furthermore, a non-empty word $w$ is said to be a \emph{super-Lyndon--Shirshov word} if it is a Lyndon--Shirshov word or a square of an odd Lyndon--Shirshov word (a square of a Lyndon--Shirshov word clearly is not a Lyndon--Shirshov word itself). 

To deal with normal forms, we invoke universal enveloping algebras. For a Lie superalgebra $\mathfrak{g}$ with generators $X$ and relations $R$, let us interpret elements of $R$ as linear combinations of commutators in the free associative algebra $\k\langle X\rangle$, so that these elements are defining relations of the universal enveloping algebra $U(\mathfrak{g})$. Then if $R$ is a Gr\"obner basis of those defining relations of~$U(\mathfrak{g})$, the Lie superalgebra $\mathfrak{g}$ has a basis whose leading terms, once we expand everything inside $U(\mathfrak{g})$ into combinations of associative monomials, are super-Lyndon--Shirshov words that are normal with respect to $G$.   

For algebras in $\TKK$, we shall always choose generators in a way compatible with the $\mathfrak{sl}_2$-weights, so that for $U=L(0)\otimes A\oplus L(2)\otimes B$ we choose a basis $\{x_i\colon i=1,\ldots,\dim(A)\}\sqcup \{e_j, f_j, h_j \colon j=1,\ldots,\dim(B)\}$ (we are going to assume our algebras finitely generated; most of our arguments require only a very mild modification for the infinite number of generators).  

\section{Free commutative associative algebras}

In this section, we shall describe free commutative associative algebras in the category $\TKK$, and show that the main conjecture does not hold for them. 

We shall start by presenting free commutative associative algebras in $\TKK$ by generators and relations. 

\begin{proposition}\label{prop:quadrelCom}
Suppose that $U=L(0)\otimes A\oplus L(2)\otimes B$. The algebra $\Com^{\TKK}(U)$ is the quotient of $S(U)$ by the relations given by the $\mathfrak{sl}_2$-submodule in $S^2(U)$ generated by $S^2(e\otimes B)$, where $e$ is the highest weight vector of $L(2)$. In particular, we have an isomorphism of commutative associative algebras
 \[
\Com^{\TKK}(U)\cong S(A)\otimes\Com^{\TKK}(L(2)\otimes B).
 \]
\end{proposition}

\begin{proof}
For a basis $\{x_i\colon i=1,\ldots,\dim(A)\}\sqcup \{e_j, f_j, h_j \colon j=1,\ldots,\dim(B)\}$ of $U$, we see that our relations must contain all monomials $e_ie_j$ (for there are no elements of weight $4$ in objects of $\TKK$), which are precisely the basis elements of the vector space $S^2(e\otimes B)$. Since we work with commutative associative algebras, the quotient by these relations is spanned by monomials that contain at most one generator $e_i$. Consequently, the quotient by the relations given by $\mathfrak{sl}_2$-submodule generated by these elements already has no elements of weight $4$ or more, and hence is an object of $\TKK$, it is the free commutative associative algebra. 

In particular, we do not impose any relations on the generators $x_i$, so 
 \[
\Com^{\TKK}(U)\cong S(A)\otimes\Com^{\TKK}(L(2)\otimes B).
 \]
\end{proof}

Proposition \ref{prop:quadrelCom} implies that in order to describe the free commutative associative algebras in $\TKK$, it is essentially enough to describe free commutative associative algebras generated by an object of the form $L(2)\otimes B$, that is by several copies of the adjoint $\mathfrak{sl}_2$-module.  We shall use Gr\"obner bases to study those algebras.

\begin{proposition}\label{prop:descriptionCom}\leavevmode
We have
 \[
\Com^{\TKK}(L(2)\otimes B)\cong L(0)\oplus L(2)\otimes B\oplus  (L(0)\otimes S^2(B)\oplus L(2)\otimes\Lambda^2(B)) \oplus L(0)\otimes \Lambda^3(B),
 \] 
where each product of more than three generators from $L(2)\otimes B$ vanishes, and products of two or three generators are given, respectively, by the map
 \[
S^2(L(2)\otimes B)\to L(0)\otimes S^2(B)\oplus L(2)\otimes \Lambda^2(B)
 \] 
sending $(u_1\otimes b_1)(u_2\otimes b_2)$ to $K(u_1,u_2)\otimes (b_1b_2)+[u_1,u_2]\otimes(b_1\wedge b_2)$ and by the map
\[
S^3(L(2)\otimes B)\to L(0)\otimes \Lambda^3(B)
 \] 
sending $(u_1\otimes b_1)(u_2\otimes b_2)(u_3\otimes b_3)$ to $K([u_1,u_2],u_3)\otimes (b_1\wedge b_2\wedge b_3)$.
\end{proposition}

\begin{proof}
According to Proposition~\ref{prop:quadrelCom}, the relations of the algebra $\Com^{\TKK}(L(2)\otimes B)$ are given by the $\mathfrak{sl}_2$-submodule generated by $S^2(e\otimes B)$. In terms of the generators $\{e_j, f_j, h_j \colon j=1,\ldots,\dim(B)\}$, this submodule has a basis of elements
 \[
e_i e_j, \ \  f_i f_j,\ \
h_i e_j+h_j e_i, \ \  h_i f_j+h_j f_i,\ \
f_i e_j+e_i f_j - h_i h_j
 \]
for all $i\le j\in \{1,\ldots,\dim(B)\}$.

Let us consider the order of generators 
 \[
e_{\dim(B)}<\ldots<e_1<h_{\dim(B)}<\ldots<h_1<f_{\dim(B)}<\ldots<f_1,
 \]
and the inverse lexicographical order associated to it; to compare two monomials in our generators with respect to that order, we choose the smallest generator in which they differ, and declare the monomials with the \emph{smaller} exponent to be \emph{larger}. Let us reproduce the above relations with their leading monomials underlined:
\begin{gather*}
\underline{e_i e_j}, \ \  \underline{f_i f_j}, i\le j\in \{1,\ldots,\dim(B)\},\\
h_i e_j+\underline{h_j e_i}, \ \  \underline{h_i f_j}+h_j f_i, i<j\in \{1,\ldots,\dim(B)\},\\
2\underline{h_i e_i}, \ \ 2\underline{h_i f_i}, i\in \{1,\ldots,\dim(B)\},\\
f_i e_j+e_i f_j - \underline{h_i h_j}, i\le j\in \{1,\ldots,\dim(B)\}.
\end{gather*}
This means that the normal quadratic monomials in this case are  
\begin{gather*}
h_ie_j, \ \ h_jf_i, i<j\in \{1,\ldots,\dim(B)\},\\
e_if_j, i,j\in \{1,\ldots,\dim(B)\}.
\end{gather*}
Consequently, the degree three monomials that are normal with respect to the leading terms of quadratic relations are precisely all the monomials 
 \[
f_ih_je_k,  i<j<k\in \{1,\ldots,\dim(B)\},
 \]
and there are no such normal monomials of degree four and higher. Moreover, since the ideal of relations of the algebra $\Com^{\TKK}(L(2)\otimes B)$ is generated by the quadratic relations above, the cosets of these normal monomials span the algebra $\Com^{\TKK}(L(2)\otimes B)$ and as such give upper bounds on dimensions of its homogeneous components. Thus, we have the following upper bounds on dimensions of nonzero homogeneous components of that algebra:
 \[
1, \dim(B), 2\binom{\dim(B)}2+\dim(B)^2, \binom{\dim(B)}3.
 \]

Let us now note that the product on 
 \[
L(0)\oplus L(2)\otimes B\oplus  (L(0)\otimes S^2(B)\oplus L(2)\otimes\Lambda^2(B)) \oplus L(0)\otimes \Lambda^3(B)
 \]
defined in the statement of the proposition is commutative and associative; the only nontrivial case is the product of three generators, where the associativity is an immediate consequence of the invariance of the Killing form. Moreover, this commutative associative algebra is generated by $L(2)\otimes B$ and thus admits a surjective map from the free algebra $\Com^{\TKK}(L(2)\otimes B)$, giving the lower bounds 
 \[
1, \dim(B), \binom{\dim(B)+1}2+3\binom{\dim(B)}2, \binom{\dim(B)}3,
 \]
on dimensions of homogeneous components of that algebra. Since we have 
 \[
\binom{\dim(B)+1}2+3\binom{\dim(B)}2=2\binom{\dim(B)}2+\dim(B)^2,
 \]
the two algebras are isomorphic.
\end{proof}

In general, computing homology of commutative algebras presented by generators and relations is not an easy task. However, we shall now see that the algebras $\Com^{\TKK}(U)$ possess sufficiently good homological properties. 

\begin{corollary}\label{cor:comKoszul}
For every object $U$ of the category $\TKK$, the free algebra $\Com^{\TKK}(U)$ is a quadratic Koszul algebra.
\end{corollary}

\begin{proof}
First of all, we recall that according to Proposition \ref{prop:quadrelCom}, we have
\[
\Com^{\TKK}(U)\cong S(A)\otimes\Com^{\TKK}(L(2)\otimes B),
 \] 
so since the symmetric algebra is well known to be Koszul and since the tensor product of Koszul algebras is Koszul \cite[Chapter 3, Corollary 1.2]{MR2177131}, it is enough to show that all algebras $\Com^{\TKK}(L(2)\otimes B)$ are Koszul. 

As we saw in the proof of Proposition \ref{prop:descriptionCom}, for a certain ordering of monomials, every algebra $\Com^{\TKK}(L(2)\otimes B)$ admits a quadratic Gr\"obner basis. It follows from \cite[Chapter 4, Sec.~8]{MR2177131} that these algebras are Koszul. 
\end{proof}

It is easy to see that the Koszul dual algebra of a commutative associative algebra $A$ is the universal enveloping algebra of a Lie superalgebra, and, if the algebra $A$ is Koszul, that latter Lie superalgebra is the cohomology of $A$, and its linear dual coalgebra is the homology of $A$. Thus, to test the main conjecture in the case of free commutative associative algebras, we should compute the underlying $\mathfrak{sl}_2$-module of the corresponding Lie superalgebra, which we shall do using Shirshov bases for Lie superalgebras. 

\begin{theorem}
Let $U=L(2)$. We have
\begin{gather*}
H_1(\Com^{\TKK}(U))^{\TKK}\cong L(2),\\ 
H_2(\Com^{\TKK}(U))^{\TKK}=0, H_3(\Com^{\TKK}(U))^{\TKK}=0, \\
H_4(\Com^{\TKK}(U))^{\TKK}\cong L(2).
\end{gather*} 
Moreover, for every object $W$ of the category $\TKK$ that contains at least one copy of the adjoint module, we have $H_4(\Com^{\TKK}(W))^{\TKK}\ne 0$. 
\end{theorem}

\begin{proof}
A direct calculation shows that the relations of the quadratic dual associative algebra of $\Com^{\TKK}(U)$ are
\begin{gather*}
\{h_p^*,f_q^*\}-\{h^*_q,f^*_p\}=0, p<q\in \{1,\ldots,\dim(B)\},\\
\{h^*_p,e^*_q\}-\{h^*_q,e^*_p\}=0, p<q\in \{1,\ldots,\dim(B)\},\\
\{e^*_p,f^*_q\}-\{e^*_q,f^*_p\}=0, p<q\in \{1,\ldots,\dim(B)\},\\
\{e^*_p,f^*_q\}+\{f^*_p,e^*_q\}+2\{h^*_p,h^*_q\}=0, p\le q\in \{1,\ldots,\dim(B)\},
\end{gather*}
where the curly brackets denote the anticommutator $\{x,y\}=xy+yx$. For $U=L_2$, this means that our algebra is generated by three elements $e^*,h^*,f^*$ subject to the only relation
 \[
\{e^*,f^*\}+\{h^*,h^*\}=0.
 \]
If we choose the order of generators for which $e^*$ is the largest one, the element $e^*f^*$ is the leading term of this relation, and, since this monomial does not form any overlaps with itself, our relation is a Gr\"obner basis. A direct inspection shows that the super-Lyndon--Shirshov words of length at most $4$ that are not divisible by $e^*f^*$ are as follows:
\begin{itemize}
\item $e^*$, $h^*$, $f^*$ of length $1$,
\item $(e^*)^2$, $e^*h^*$, $(h^*)^2$, $h^*f^*$, $(f^*)^2$ of length $2$, 
\item $(e^*)^2h^*$, $e^*(h^*)^2$, $e^*h^*f^*$, $(h^*)^2f^*$, $h^*(f^*)^2$ of length $3$,
\item $(e^*)^3h^*$, $(e^*)^2(h^*)^2$, $(e^*)^2h^*f^*$, $e^*(h^*)^3$, $e^*(h^*)^2f^*$, $e^*h^*f^*h^*$, $e^*h^*(f^*)^2$, $(h^*)^3f^*$, $(h^*)^2(f^*)^2$, $h^*(f^*)^3$ of length $4$.
\end{itemize}
The words of length $2$ are immediate to list. For lengths $3$ and $4$, it is sufficiently easy, and requires only very basic observations; for instance, if a super-Lyndon--Shirshov word starts with $h^*$, it cannot contain $e^*$, for otherwise it cannot be the largest among its cyclic shifts.  
Computing the weights of these words, we find 
\begin{gather*}
H_1(\Com^{\TKK}(U))\cong L(2),\\ 
H_2(\Com^{\TKK}(U))\cong L(4),\quad
 H_3(\Com^{\TKK}(U))\cong L(4), \\
 H_4(\Com^{\TKK}(U))\cong L(6)\oplus L(2),
\end{gather*} 
and the first statement follows. 

Suppose now that $W$ is an object of $\TKK$ that contains at least one copy of $L(2)$. Let $A$ be a subalgebra of $\Com^{\TKK}(W)$ generated by that copy; clearly, $A\cong \Com^{\TKK}(U)$, and the bar complex of $\Com^{\TKK}(W)$ contains the bar complex of $\Com^{\TKK}(U)$ as a direct summand, so we have $H_4(\Com^{\TKK}(U))^{\TKK}\ne 0$, as required. 
\end{proof}

\begin{remark}
Note that since our algebras are Koszul, we can easily determine homology from its $\mathfrak{sl}_2$-character, which in turn can be determined from the general properties of Koszul algebras. We did the corresponding calculations using \texttt{sage} \cite{sagemath} and obtained the  following $GL(B)$-module isomorphisms:
\begin{gather*}
H_1(\Com^{\TKK}(L(2)\otimes B))^{\TKK}\cong L(2)\otimes B,\\ 
H_2(\Com^{\TKK}(L(2)\otimes B))^{\TKK}\cong 0, H_3(\Com^{\TKK}(L(2)\otimes B))^{\TKK}\cong 0, \\
H_4(\Com^{\TKK}(L(2)\otimes B))^{\TKK}\cong L(2)\otimes S^4(B),\\
H_5(\Com^{\TKK}(L(2)\otimes B))^{\TKK}\cong L(2)\otimes (S^5(B)\oplus S^{4,1}(B)\oplus S^{3,2}(B))\oplus L(0)\otimes S^{4,1}(B).
\end{gather*} 
Here, for a partition $\lambda\vdash n$, we denote by $S^\lambda(B)$ the corresponding Schur functor \cite[Chapter 1]{MR3443860}, that is the $GL(B)$-module on the multiplicity of the irreducible $S_n$-module corresponding to $\lambda$ in the tensor product $B^{\otimes n}$. 

Moreover, $H_6(\Com^{\TKK}(L(2)\otimes B))^{\TKK}$ contains the trivial $\mathfrak{sl}_2$-module with the multiplicity $S^6(B)\oplus S^{5,1}(B)\oplus S^{4,2}(B)\oplus S^{4,2}(B)\oplus S^{3,2,1}(B)$, so the trivial $\mathfrak{sl}_2$-module also appears in the higher homology, as long as our object contains at least one copy of $L(2)$. 
\end{remark}

\section{Free associative algebras}

In this section, we shall describe free associative algebras in the category $\TKK$, and show that the main conjecture holds for them. 

We shall start with the following analogue of Proposition \ref{prop:quadrelCom} which lists all relations of the free algebra. 

\begin{proposition}\label{prop:quadrelAss}
Suppose that $U=L(0)\otimes A\oplus L(2)\otimes B$. The algebra $\Ass^{\TKK}(U)$ is the quotient of $\Ass(U)$ by the relations given by the $\mathfrak{sl}_2$-submodule in $U^{\otimes 2}$ generated by $(e\otimes B)\otimes T(L(0)\otimes A)\otimes (e\otimes B)$, where $e$ is the highest weight vector of $L(2)$. 
\end{proposition}

\begin{proof}
For a basis $\{x_i\colon i=1,\ldots,\dim(A)\}\sqcup \{e_j, f_j, h_j \colon j=1,\ldots,\dim(B)\}$ of the vector space $U=L(0)\otimes A\oplus L(2)\otimes B$, we see that our relations must contain all monomials $e_ix_{k_1}\cdots x_{k_s}e_j$ (for there are no elements of weight $4$ in objects of $\TKK$), which are precisely the basis elements of the vector space $(e\otimes B)\otimes T(L(0)\otimes A)\otimes (e\otimes B)$. 

It is easy to describe the $\mathfrak{sl}_2$-submodule generated by each such element: it consists of the elements 
\begin{gather*}
e_i \mathbf{x} e_j,\ \  f_i \mathbf{x} f_j,\\
h_i \mathbf{x} e_j+e_i \mathbf{x} h_j, \ \  h_i \mathbf{x} f_j+f_i \mathbf{x} h_j,\\
f_i \mathbf{x} e_j+e_i \mathbf{x} f_j - h_i \mathbf{x} h_j, 
\end{gather*}
for $i,j\in \{1,\ldots,\dim(B)\}$, where we denote for brevity $\mathbf{x}=x_{k_1}\cdots x_{k_s}$. Let us choose some order of generators for which all generators $e_j$ are greater than all generators $h_j$, which in turn are greater than all generators $f_j$, which are greater than all generators $x_i$, and consider the graded lexicographic order of monomials. Then the leading terms of those relations are
 \[
e_i \mathbf{x} e_j, f_i \mathbf{x} f_j, e_i \mathbf{x} h_j, h_i \mathbf{x} f_j, e_i \mathbf{x} f_j
 \]
for $i,j\in \{1,\ldots,\dim(B)\}$. In principle, we do not know whether our relations form a Gr\"obner basis, but the monomials that are normal with respect to them form a spanning set in the quotient algebra, and so we can look at them to ensure that the quotient by our relations is an object of $\TKK$. Indeed, such a normal monomial can contain at most one element $e_j$ (since no element from $L(2)\otimes B$ can follow $e_j$, even separated by some factors $x_i$), and at most one element $f_j$ (since no element from $L(2)\otimes B$ can precede $f_j$, even separated by some factors $x_i$).  
Consequently, the quotient by the relations given by $\mathfrak{sl}_2$-submodule generated by these elements already has no elements of weight $4$ or more, and hence is an object of $\TKK$, it is the free associative algebra. 
\end{proof}

We shall now use the calculation from the previous proof to give an explicit description of the algebra $\Ass^{\TKK}(L(2)\otimes B)$.

\begin{proposition}\label{prop:descriptionAss}
We have 
 \[
\Ass^{\TKK}(L(2)\otimes B)=\k\oplus L(2)\otimes B\oplus\bigoplus_{p\ge 2}\Mat_2(B^{\otimes p})
 \]
where the product is given by simultaneously performing the matrix product and the concatenation of tensors.  
\end{proposition}

\begin{proof}
Note that the proof of Proposition \ref{prop:quadrelAss} shows that for any order of generators for which all generators $e_j$ are greater than all generators $h_j$, which in turn are greater than all generators $f_j$, the defining relations of the algebra $\Ass^{\TKK}(L(2)\otimes B)$ have the following leading terms:
 \[
e_ie_j,\, f_if_j,\, e_ih_j,\, h_if_j,\, e_if_j.
 \]
Thus, for $p\ge 2$, the monomials of degree $p$ that are normal forms with respect to these relations are of the following types:
 \[
h_{i_1}h_{i_2}\cdots h_{i_p},\quad
h_{i_1}h_{i_2}\cdots h_{i_{p-1}}e_{i_p},\quad
f_{i_1}h_{i_2}\cdots h_{i_p},\quad
f_{i_1}h_{i_2}\cdots h_{i_{p-1}}e_{i_p}
 \]
(there can be no elements after $e_j$ and no elements before $f_j$). This gives an estimate from the above on dimensions of homogeneous components of the algebra $\Ass^{\TKK}(L(2)\otimes B)$: these normal monomials form a spanning set. At the same time, the vector space 
 \[
\k\oplus L(2)\otimes B\oplus\bigoplus_{p\ge 2}\Mat_2(B^{\otimes p})
 \]
with the algebra structure described above is an algebra in the category $\TKK$ (since we have an isomorphism of $\mathfrak{sl}_2$-modules $\Mat_2\cong L(0)\oplus L(2)$) and is easily seen to be generated by $L(2)\otimes B$, so we obtain a lower bound that coincides with the upper one, and therefore our algebra is free, and our relations form a Gr\"obner basis. 
\end{proof}

\begin{corollary}
The algebra $\Ass^{\TKK}(L(2)\otimes B)$ is Koszul.
\end{corollary}

\begin{proof}
This follows from the fact that this algebra has a quadratic Gr\"obner basis, since this implies the Koszul property \cite[Chapter 4, Theorem~3.1]{MR2177131}.
\end{proof}

\begin{remark}
If one considers the general case of the free associative algebra generated by $U=L(0)\otimes A\oplus L(2)\otimes B$, that algebra, as we saw in the proof of Proposition \ref{prop:quadrelAss}, is not at all the coproduct of two free algebras, and is not Koszul (since it is not even quadratic). It is however possible to write down a description of this algebra analogous to that of Proposition \ref{prop:descriptionAss}. That description, which we shall not really use in this article, is
 \[
\Ass^{\TKK}(U)=T(A)\oplus L(2)\otimes (T(A)\otimes B\otimes T(A))\oplus\bigoplus_{p\ge 2}\Mat_2\left((T(A)\otimes B\otimes T(A))^{\boxtimes p}\right)
 \]
where $\boxtimes = \otimes_{T(A)}$, and the product is given by simultaneously performing the matrix product and the concatenation of tensors.
\end{remark}

We shall now prove that the main conjecture holds for associative algebras in the category $\TKK$.

\begin{theorem}\label{th:vanishAss}
For every object $U$ of $\TKK$, the truncated homology $H_\bullet(\Ass^{\TKK}(U))^{\TKK}$ vanishes in each degree greater than one.
\end{theorem}

\begin{proof}
For an augmented associative algebra $R$, the homology $H_\bullet(R)$ is the homology of the bar construction 
$\mathsf{B}_{\Ass}(R)$, which computes $\Tor_{\bullet}^{R}(\k,\k)$. That $\Tor$ functor can be computed using any free resolution 
$(F_\bullet,d)=(X_\bullet\otimes R,d)$ of the augmentation module $\k$; precisely, 
 \[
\Tor_{\bullet}^{R}(\k,\k)\cong H_{\bullet}(F_\bullet\otimes_R\k,d\otimes_R 1)\cong H_{\bullet}((X_\bullet\otimes R)\otimes_R\k,d\otimes_R 1)\cong H_{\bullet}(X_\bullet,\overline{d}),
 \]
where $\overline{d}$ is the part of $d$ that ``survives'' after tensoring with the augmentation module. 

We shall use the so called Anick resolution \cite{MR0846601,MR1360005} that exists for an algebra with a given Gr\"obner basis. Let us briefly recall its key features. Suppose $R=\k\langle X\mid G \rangle$ is an augmented associative algebra with generators $X$ and relations $R$. We shall assume that $G$ is a reduced Gr\"obner basis for a certain admissible order of monomials in generators $X$ (this means that the leading terms of $G$ form an antichain in the poset of monomials in $X$ with the order given by divisibility). Out of that datum, one can inductively define the notion of a (right) $k$-chain and a tail of a $k$-chain as follows.
\begin{itemize}
\item the $0$-chains are generators (elements of $X$); each of them coincides with its tail;
\item for $k>0$, a $k$-chain is a word $c$ in the alphabet $X$ which can be written as the concatenation $c't$, where $c'$ is a $(k-1)$-chain, and $t$ is some word, called the tail of $c$, such that
\begin{itemize}
\item if we denote by $t'$ the tail of $c'$, there exists a factorization $t'=m_1 m_2$ such that $m_2 t$ is the leading term of an element of $G$, and there are no other divisors of $t't$ that are leading terms of $G$,
\item no proper beginning of $c$ is a $k$-chain. 
\end{itemize}
\end{itemize}
It is established in \cite{MR0846601} that there exists a resolution of the (right) augmentation $R$-module $\k$ by the free right $R$-modules of the form
 \[
\ldots \to\k C_k\otimes R\xrightarrow{d_k}\k C_{k-1}\otimes R\to\ldots\to\k C_1\otimes R\xrightarrow{d_1} \k C_0\otimes R\xrightarrow{d_0} R\to0
 \]
The differentials $d_k$ of that resolution are constructed inductively together with the splittings 
$i_k\colon\operatorname{Ker} d_{k-1}\to \k C_k\otimes R$ satisfying $d_k i_k=\operatorname{Id}_{|\operatorname{Ker}d_{k-1}}$.

In our case, $R=\Ass^{\TKK}(U)$ for some object $U$ of $\TKK$. Let us first consider the particular case $U=L(2)\otimes B$, in which we established that the algebra has a quadratic Gr\"obner basis. This means that each $k$-chain is a word of length $k+1$, for the only way to add a tail to a $k$-chain by ``linking'' it with a quadratic relation is when that tail consists of just one letter. Therefore, for each $c\in C_k\otimes R$, we have $d_k(c)\in C_{k-1}\otimes \overline{R}$, where $\overline{R}$ is the augmentation ideal of $R$, and so $\overline{d}_k(c)=0$. Thus, 
 \[
H_k(\Ass^{\TKK}(L(2)\otimes B))\cong \k C_{k-1}.
 \]
Moreover, the leading terms of the Gr\"obner basis found in Proposition \ref{prop:descriptionAss} are
 \[
e_i e_j,  f_i f_j, e_i h_j,  h_i f_j, e_i f_j,
 \]
so each $k$-chain is of one the following forms:
\begin{gather*}
e_{i_1}e_{i_2}\cdots e_{i_s}f_{j_1}f_{j_2}\cdots f_{j_t}, s,t\ge 0, s+t=k+1\\
e_{i_1}e_{i_2}\cdots e_{i_s}h_p f_{j_1}f_{j_2}\cdots f_{j_t}, s,t\ge 0, s+t=k.
\end{gather*}
Every element of the first type has the $\mathfrak{sl}_2$-weight $2s-2(k+1-s)=4s-2-2k$, and every element of the second type has the $\mathfrak{sl}_2$-weight $2s-2(k-s)=4s-2k$. It is therefore clear that for the fixed sequence of indices in $\{1,\ldots,n\}^{k+1}$ these weights contain every even number between $2k+2$ and $-2k-2$ exactly once, and we have an $\mathfrak{sl}_2$-module isomorphism
 \[
\k C_k\cong L(2k+2)^{n^{k+1}}.
 \]
Thus, $H_\bullet(\Ass^{\TKK}(L(2)\otimes B))^{\TKK}=C_{\bullet-1}^{\TKK}$ 
vanishes in degrees greater than $1$. 

Let us now discuss the general case $R=\Ass^{\TKK}(L(0)\otimes A\oplus L(2)\otimes B)$. As established in Proposition \ref{prop:quadrelAss}, while the relations of this algebra are not quadratic, they are ``quadratic relative to the base algebra'' $T(A)$: they are the elements
\begin{gather*}
e_i \mathbf{x} e_j,\ \  f_i \mathbf{x} f_j,\\
h_i \mathbf{x} e_j+e_i \mathbf{x} h_j, \ \  h_i \mathbf{x} f_j+f_i \mathbf{x} h_j,\\
f_i \mathbf{x} e_j+e_i \mathbf{x} f_j - h_i \mathbf{x} h_j, 
\end{gather*}
where we denote for brevity $\mathbf{x}=x_{k_1}\cdots x_{k_s}$. Let us show that for any order of generators for which all generators $e_j$ are greater than all generators $h_j$, which in turn are greater than all generators $f_j$, which are greater than all generators $x_i$, these relations form a Gr\"obner basis for the graded lexicographic order of monomials. The fastest way to justify this is to say that every S-polynomial of two elements of this set will be reduced to zero in the exact same way as every S-polynomial of two elements of the set of relations for $\Ass^{\TKK}(L(2)\otimes B)$: the fact that now we have some ``layers'' $x_{k_1}\cdots x_{k_s}$ between the variables from the adjoint representation does not change anything in computing the reductions. In fact, the same argument applies to the Anick resolution: instead of each single chain, we shall have infinitely many chains with ``layers'' $x_{k_1}\cdots x_{k_s}$ between the variables from the adjoint representation. Computation of the $\mathfrak{sl}_2$-weights does not change, and we find that $H_k(\Ass^{\TKK}(L(0)\otimes A\oplus L(2)\otimes B))$ is a multiple of $L(2k)$, completing the proof.
\end{proof}

\section{Free Lie algebras}

In this section, we discuss the original case of the main conjecture: that for free Lie algebras. That conjecture remains out of reach for the time being, but we hope that recording various results and non-results should be helpful in attacking it. 

\subsection{Universal enveloping algebras: the Poincar\'e--Birkhoff--Witt non-theorem}

If the category $\TKK$ were symmetric monoidal, we would be able to compute Lie algebra homology using $\Tor$ over the universal enveloping algebra. Let us consider the simplest possible example which exhibits a surprising numerical coincidence that however does not present itself in more complicated situations. 

\begin{example}\label{ex:GL}
Consider the case of the Tits--Allison--Gao construction applied to $t\k[t]$, the free Jordan algebra on one generator. It is explained in \cite[Sec.~4.1]{MR4235202} that in this case one of the theorems of Garland and Lepowsky on Lie algebra homology \cite{MR0414645} applies, showing that $H_k(\Lie^{\TKK}(L(2)))\cong L(2k)$. Similarly, examining the proof of Theorem \ref{th:vanishAss}, we find
$H_k(\Ass^{\TKK}(L(2)))\cong L(2k)$.  
Thus, in this case computing the homology using the universal enveloping algebra ``works'', in the weakest possible sense of giving the correct answer. However, this phenomenon does not persist: already for the free Jordan algebra on one generator, the situation is much more complicated. To see that, recall that we have seen in the proof of Theorem~\ref{th:vanishAss} that we have a $\mathfrak{sl}_2$-module isomorphism
 \[
H_2(\Ass^{\TKK}(L(2)\oplus L(2)))\cong L(4)^4
 \]
and that $L(4)$ appears only in homological degree two. However, since a theorem of Shirshov \cite{MR668355} asserts that the free Jordan algebra on two generators is special, that is, coincides with the Jordan subalgebra of $\k\langle x,y\rangle$ generated by $x,y$ under the Jordan product defined by Formula \eqref{eq:SJordan}; thus, it is not hard to compute the Euler characteristic of the $\mathfrak{sl}_2$-action on the Chevalley--Eilenberg complex of the Lie algebra $\Lie^{\TKK}(L(2)\oplus L(2))$ and to see that $L(4)$ appears in the homology with multiplicity at least $10$. 
\end{example}

In the remainder of this section, we shall explain in what precise sense the universal enveloping algebras in $\TKK$ lose some information about their Lie algebras: it turns out that they do not generally contain the original Lie algebra as a subalgebra. We shall focus on universal enveloping algebras of free Lie algebras in $\TKK$, which, of course, are free associative algebras in $\TKK$ (the composition of two left adjoint functors is a left adjoint functor). Note that, as established by the first author and Tamaroff \cite{MR4300233}, under a certain functoriality assumption, a Poincar\'e--Birkhoff--Witt type theorem holds for every Lie algebra if and only if it holds for all free Lie algebras.

\begin{theorem}
The Lie subalgebra $\mathfrak{L}(B)$ of the algebra $\Ass^{\TKK}(L(2)\otimes B)$ generated by $L(2)\otimes B$ is isomorphic to $\TAG(\SJord(B))$,
the Tits--Allison--Gao construction of the special Jordan algebra generated by $B$.  
\end{theorem}

\begin{proof}
It is proved in \cite[Corollary 2]{MR4235202} that we have 
$\calB(\SJord(B))\cong \Inner(\SJord(B))$,
so that the Tits--Allison--Gao functor applied to the free special Jordan algebra gives the same result as the Tits--Kantor--Koecher construction applied to that algebra, so we may focus on the latter.

According to \cite[Lemma 5]{MR4235202}, we have $\Inner(\SJord(B))\cong [\SJord(B),\SJord(B)]$ (where the commutator is taken in the free associative algebra), and thus the Tits--Kantor--Koecher construction applied to the free special Jordan algebra is
 \[
L(0)\otimes [\SJord(B),\SJord(B)]\oplus L(2)\otimes \SJord(B),  
 \]
with the Lie bracket given by 
\begin{multline*}
[u_1+a_1\otimes j_1,u_2+a_2\otimes j_2]= \\
\left([u_1,u_2]+K(a_1,a_2)[j_1,j_2]\right)+\left(a_2\otimes u_1(j_2)-a_1\otimes u_2(j_1)
+ [a_1,a_2]\otimes j_1\circ j_2\right).
\end{multline*}  
This means that if we denote by $1$ the basis element of $L(0)$ and by $e,h,f$ the bases elements of $L(2)\cong\mathfrak{sl}_2$, the Lie bracket is given by
\begin{gather*}
[1\otimes u_1,1\otimes u_2]=1\otimes [u_1,u_2],\\
[I\otimes u,e\otimes j]=e\otimes [u,j],\\
[I\otimes u,h\otimes j]=h\otimes [u,j],\\
[I\otimes u,f\otimes j]=f\otimes [u,j],\\
[e\otimes j_1,e\otimes j_2]=0\\
[e\otimes j_1,f\otimes j_2]=h\otimes (j_1\circ j_2)+1\otimes[j_1,j_2],\\
[f\otimes j_1,f\otimes j_2]=0\\
[e\otimes j_1,h\otimes j_2]=-2e\otimes (j_1\circ j_2),\\
[h\otimes j_1,f\otimes j_2]=2f\otimes (j_1\circ j_2),\\
[h\otimes j_1,h\otimes j_2]=\frac12\otimes[j_1,j_2]
\end{gather*}  

On the other hand, we recall from Proposition \ref{prop:descriptionAss} that the algebra 
 \[
\Ass^{\TKK}(L(2)\otimes B)
 \]
is explicitly given by 
 \[
\k\oplus L(2)\otimes B\oplus\bigoplus_{p\ge 2}\Mat_2(B^{\otimes p}).
 \]
It will be convenient to write elements of the latter algebra as combinations of the elements of the form
 \[
E\otimes \mathbf{x}, F\otimes \mathbf{x}, H\otimes \mathbf{x}, I\otimes \mathbf{x},
 \]
where $E,F,H$ correspond to the standard embedding of $\mathfrak{sl}_2$ into $2\times 2$-matrices, $I$ is the identity $2\times 2$-matrix, $\mathbf{x}$ is an word in basis elements of $B$ (for $I\otimes \mathbf{x}$, this word should be of length at least $2$). Let us compute the Lie brackets of such elements, taking into the multiplication table for the matrices $E,F,H$:
\begin{gather*}
EF=\frac12(I+H), \ \ FE=\frac12(I-H),\\
EH=-E, \ \ HE = E, \ \ FH=F, \ \ HF=-F,\\
E^2=F^2=0, \ \ H^2=I.
\end{gather*}
We obtain the following:
\begin{gather*}
[E\otimes \mathbf{x},E\otimes \mathbf{x}']=[F\otimes \mathbf{x},F\otimes \mathbf{x}']=0,\\
[E\otimes \mathbf{x},F\otimes \mathbf{x}']=EF\otimes \mathbf{x}\mathbf{x}'-FE\otimes \mathbf{x}'\mathbf{x}=H\otimes(\mathbf{x}\circ\mathbf{x}')+I\otimes\frac12[\mathbf{x},\mathbf{x}'],\label{eq:h-jord}\\
[E\otimes \mathbf{x},H\otimes \mathbf{x}']=EH\otimes\mathbf{x}\mathbf{x}'-HE\otimes\mathbf{x}'\mathbf{x}=-2E\otimes (\mathbf{x}\circ\mathbf{x}'),\label{eq:e-jord}\\
[I\otimes \mathbf{x},E\otimes \mathbf{x}']=E\otimes[\mathbf{x},\mathbf{x}'],\\
[F\otimes \mathbf{x},H\otimes \mathbf{x}']=FH\otimes\mathbf{x}\mathbf{x}'-HF\otimes\mathbf{x}'\mathbf{x}=2F\otimes (\mathbf{x}\circ\mathbf{x}'),\label{eq:f-jord}\\
[I\otimes \mathbf{x},F\otimes \mathbf{x}']=F\otimes[\mathbf{x},\mathbf{x}'],\\
[H\otimes \mathbf{x},H\otimes \mathbf{x}']=H^2\otimes\mathbf{x}\mathbf{x}'-H^2\otimes\mathbf{x}'\mathbf{x}=I\otimes [\mathbf{x},\mathbf{x}'],\label{eq:i-commutator}\\
[I\otimes \mathbf{x},H\otimes \mathbf{x}']=H\otimes[\mathbf{x},\mathbf{x}'],\\
[I\otimes \mathbf{x},I\otimes \mathbf{x}']=I\otimes[\mathbf{x},\mathbf{x}'].
\end{gather*}

Examining these formulas, we note that they match precisely the Lie brackets of the Tits--Kantor--Koecher construction above. This now allows us to prove by induction on degree of elements (with respect to $B$) that the identity map of the vector space $L(2)\otimes B$ the Lie algebra map extends to a well defined Lie algebra map from 
 \[
\TAG(\SJord(B))\cong L(0)\otimes [\SJord(B),\SJord(B)]\oplus L(2)\otimes \SJord(B)  
 \]
to the Lie subalgebra of $\Ass^{\TKK}(L(2)\otimes B)$ generated by $L(2)\otimes B$. This statement is trivial for generators ($n=1$), and the coincidence of the above formulas assures the step of induction.
\end{proof}

\subsection{Some other observations}

\subsubsection{Computational evidence}

In \cite{MR4235202}, some computational evidence is offered in favour of the main conjecture in the case of free Lie algebras. One computation that goes slightly further was performed while preparing this paper.

\begin{proposition}
The multilinear component of the free Jordan algebra on $8$ generators (in other words, the component $\Jord(8)$ of the Jordan operad) has dimension $19089$. Moreover, this agrees with the dimension of the degree $8$ coefficient of the Schur functor of the operad $\Jord$ predicted by the main conjecture.   
\end{proposition}

\begin{proof}
The dimension of $\Jord(8)$ is computed using the \texttt{albert} software \cite{10.1145/190347.190358}. The degree $8$ coefficient of the Schur functor of $\Jord$ predicted by the main conjecture is computed by implementing the recursive formula of the proof of \cite[Lemma 1]{MR4235202} in \texttt{sage} \cite{sagemath}.
\end{proof}

\subsubsection{Braided categories of the same size}

As we repeatedly emphasized, if the category $\TKK$ were symmetric monoidal, the main conjecture would have been trivial. However, it is not: braided monoidal categories with two objects $1,c$ satisfying $c\otimes c=1\oplus c$ can be classified, and neither of them is symmetric. Algebras in those categories that are braided commutative can be classified \cite{MR2889539}, but nothing is understood about braided Lie algebras in those categories, which would be an interesting question to investigate, though we doubt that it would help with the main conjecture.

\subsubsection{Koszul duality}

If our category were symmetric monoidal, the main conjecture would be equivalent to the Koszulness of the appropriate operad. In such situation, proving that conjecture for Lie algebras and commutative associative algebras would be equivalent tasks. In our setting, the same is not at all clear, and we do not dare to see what one can learn from the fact that the main conjecture does not hold for free commutative associative algebras.

\subsubsection{The case of superalgebras}

In \cite{Shang2022}, the natural analogue of the main conjecture of \cite{MR4235202} is explored. While the latter conjecture is stated on the level of Schur functors, and therefore its validity for free algebras is manifestly equivalent to its validity for superalgebras, that work raises an interesting question of studying the superized version of the category $\TKK$, where the usual interplay between algebras and superalgebras (see, e.g. \cite{MR1156760,MR1492063,MR0981830}) is destroyed due to the lack of a symmetric monoidal structure.

\section*{Funding} Various stages of this project were supported by the S\~{a}o Paulo research foundation (FAPESP grant 2022/10933-3), by the French national research agency (project ANR-20-CE40-0016), and by Institut Universitaire de France. 

\section*{Acknowledgements} We thank Ivan Shestakov for many useful discussions. The first author is grateful to Irvin Hentzel for independently confirming the dimension of $\Jord(8)$.    

\section*{Conflict of interest} The authors has no conflicts of interest to declare that are relevant to the content of this article.

\printbibliography
\end{document}